\documentclass{amsart}

\usepackage{amsmath,amsxtra,amsthm,amssymb,amsfonts} 

\theoremstyle{plain}

\newtheorem{thm}{Theorem}[section]
\newtheorem*{thmnn}{Theorem}

\newtheorem{defi}[thm]{Definition}
\newtheorem{prop}[thm]{Proposition}
\newtheorem{lem}[thm]{Lemma}
\newtheorem{cor}[thm]{Corollary}

\theoremstyle{definition}

\newtheorem{rem}[thm]{Remark}
\newtheorem{exa}[thm]{Example}

\def\nat{\mathbb{N}}
\def\rls{\mathbb{R}}
\def\exrls{(-\infty,\infty]}
\def\eps{\varepsilon}
\def\wto{\stackrel{w}{\to}}
\def\ol{\overline}
\def\clco{\operatorname{\ol{co}}}
\def\argmin{\operatornamewithlimits{\arg\min}}

\def\dom{\operatorname{dom}}

\def\cldom{\operatorname{\ol{dom}}}

\def\calx{\mathcal{X}}

\def\hs{\mathcal{H}}

\def\as{\mathrel{\mathop:}=}          

\begin{document}
\title[Convergence of semigroups]{Convergence of nonlinear semigroups under nonpositive curvature}
\author[M. Ba\v{c}\'ak]{Miroslav Ba\v{c}\'ak}
\date{\today}
\thanks{The research leading to these results has received funding from the
 European Research Council under the European Union's Seventh Framework
 Programme (FP7/2007-2013) / ERC grant agreement no 267087.}
\address{Miroslav Ba\v{c}\'ak, Max Planck Institute, Inselstr. 22, 04 103 Leipzig, Germany}
\email{bacak@mis.mpg.de}

\begin{abstract}
The present paper is devoted to gradient flow semigroups of convex functionals on Hadamard spaces. We show that the Mosco convergence of a sequence of convex lsc functions implies convergence of the corresponding resolvents and convergence of the gradient flow semigroups. This extends the classical results of Attouch, Brezis and Pazy into spaces with no linear structure. The same method can be further used to show the convergence of semigroups on a~\emph{sequence} of spaces, which solves a problem of [Kuwae and Shioya, Trans. Amer. Math. Soc., 2008].
\end{abstract}

\maketitle
\section{Introduction}

Given a maximal monotone operator $A:H\to 2^H$ on a Hilbert space $H,$ the following parabolic problem
\begin{align*}
 \dot{u}(t) & \in - Au(t),\quad t\in(0,\infty), \\
 u(0) & =u_0\in H,
\end{align*}
for a curve $u:[0,\infty)\to H$ has been studied for a long time, and is quite well-understood, see for instance~\cite{brezis-b}. The most important case occurs when $A$ is the (convex) subdifferential of a convex lsc function $f:H\to\exrls,$ that is, $A=\partial f,$ and we hence have 
\begin{equation} \label{eq:gradflowproblem}
\dot{u}(t)\in - \partial f\left(u(t)\right),\quad t\in(0,\infty).
\end{equation}
This is a nonsmooth analog of the gradient flow equation. Even though the above problem~\eqref{eq:gradflowproblem} is nonlinear due to the operator $\partial f,$ it has been so far considered mostly on \emph{linear} spaces. Recent years, however, have witnessed a great deal of interest in gradient flows in metric spaces, and their applications to optimal transport, PDEs, and probability theory~\cite{ambrosio}.

In the present paper we study the gradient flow problem in metric spaces of nonpositive curvature in the sense of Alexandrov, so-called Hadamard spaces, and therefore continue along the lines of~\cite{ppa,jost94,jost95,jost97,jost-ch,japan,mayer,stoj}. There have been also many related results in some special instances of Hadamard spaces, namely, in manifolds of nonpositive sectional curvature (\cite{sevilla,sevilla2,papa}), and the Hilbert ball (\cite{kop} and the references therein). On the other hand some parts of this theory can be extended into more general metric spaces~\cite{genaro}.

Since Hadamard spaces allow for a natural notion of convexity, and have many pleasant properties, for instance metric projections onto convex closed sets are nonexpansive, they appear a very convenient framework for (convex) analysis indeed. In particular, the results of the above references show, in a sense, that it is not the linear structure of the underlying space, but rather its nonpositive curvature, what enables the convergence of the resolvents, gradient flow semigroups, and the proximal point algorithm.

Before introducing our results, let us first describe the geometrical structure of these spaces with a special regard to convexity. For the notation and definitions not explained here, the reader is referred to Sections~\ref{sec:pre} and~\ref{sec:weak}.

\subsection{Hadamard spaces and convex functions}
If a geodesic metric space $(X,d)$ satisfies the following inequality 
\begin{equation} \label{eq:cat}
 d\left(x,\gamma(t)\right)^2\leq (1-t)d\left(x,\gamma(a)\right)^2+td\left(x,\gamma(b)\right)^2-t(1-t)d\left(\gamma(a),\gamma(b)\right)^2,
\end{equation}
for any $x\in X,$ any geodesic $\gamma:[a,b]\to X,$ and any $t\in[0,1],$ we say it has nonpositive curvature (in the sense of Alexandrov), or that it is a CAT(0) space. A complete CAT(0) space is called an \emph{Hadamard space.}

The class of Hadamard spaces includes Hilbert spaces, $\rls$-trees, Euclidean Bruhat-Tits buildings, classical hyperbolic spaces, complete simply connected Riemannian manifolds of nonpositive sectional curvature, the Hilbert ball, CAT(0) complexes, and many other important spaces included in none of the above classes~\cite{book}.

Recall that a function $f:\hs\to\exrls,$ is \emph{convex} if, for any geodesic $\gamma:[0,1]\to\hs,$ the function $f\circ\gamma$ is convex. Here we collect several important instances of convex functions. In all examples we assume $(\hs,d)$ is an Hadamard space.
\begin{exa}[Indicator functions] \label{exa:indicator}
Let $C\subset \hs$ be a convex set. Define the~\emph{indicator function} of~$C$ by
$$ \iota_C(x)\as\left\{
\begin{array}{ll} 0, & \text{if } x\in C,  \\ \infty,  & \text{if } x\notin C. \end{array} \right. $$
Then $\iota_C$ is a convex function, and it is lsc if and only if $C$ is closed.
\end{exa}

\begin{exa}[Distance functions] \label{exa:dist}
Given $x_0\in\hs,$ the function
\begin{equation}\label{eq:dist} x\mapsto d\left(x,x_0\right),  \quad x\in \hs,\end{equation}
 is convex and continuous. The function $d\left(\cdot,x_0\right)^p$ for $p>1$ is strictly convex, and if $p=2,$ this function is even \emph{strongly} convex, see Remark~\ref{rem:stronglyconvex}. More generally, the \emph{distance function} to a closed convex subset $C\subset \hs,$ defined as
$$d_C(x)\as\inf_{c\in C} d(x,c),  \quad x\in \hs,$$
is convex and $1$-Lipschitz \cite[Proposition~2.4, p.176]{book}. 
\end{exa}

\begin{exa} \label{exa:sumofdtothep}
 Given a finite number of points $a_1,\dots,a_N\in \hs,$ and positive weights $w_1,\dots,w_N$ with $\sum_{n=1}^N w_n=1,$ we consider the function
$$ f(x)= \sum_{n=1}^N w_n d\left(x,a_n\right)^p,\qquad x\in\hs,$$
where $p\in[1,\infty).$ The function~$f$ is convex continuous and plays an important role in optimization. We study minimizers of this function in two important cases:
\begin{enumerate}
 \item If $p=1,$ then $f$ becomes the objective function in the \emph{Fermat-Weber problem} for optimal facility location. If, moreover, all the weights $w_n=\frac1N,$ a minimizer is called a \emph{median} of the points $a_1,\dots,a_N.$
 \item If $p=2,$ then a minimizer of~$f$ is the \emph{barycenter} of the probability measure
$$ \mu=\sum_{n=1}^N w_n \delta_{a_n},$$
where $\delta_{a_n}$ stands for the Dirac measure at the point $a_n.$ For further details on barycenters, the reader is referred to~\cite[Chapter 3]{jost2} and~\cite{sturm}. If, moreover, all the weights $w_n=\frac1N,$ the (unique) minimizer is called the \emph{mean} of the points $a_1,\dots,a_N.$
\end{enumerate}
\end{exa}

\begin{exa}\label{exa:limsup}
Let $(x_n)\subset \hs$ be a bounded sequence. Define the function $\omega:\hs\to[0,\infty)$ as
$$\omega\left(x,\left(x_n\right)\right)\as\limsup_{n\to\infty} d(x,x_n)^2,\quad x\in \hs.$$
It is locally Lipschitz, because $d(\cdot,x_n)^2$ are locally Lipschitz with a common Lipschitz constant for all $n\in\nat.$ Lemma~\ref{lem:limsup} states that $\omega$ is strongly convex.
This function will be used in the definition of the weak convergence; see~\eqref{eq:omega}.
\end{exa}

\begin{exa}[Displacement functions]
Let $T:\hs\to \hs$ be an isometry. The \emph{displacement function} of $T$ is the function $\delta_T:\hs\to[0,\infty)$ defined by
$$\delta_T(x)\as d(x,Tx),\quad x\in \hs.$$
It is convex and $2$-Lipschitz \cite[Definition~II.6.1]{book}.
\end{exa}

\begin{exa}[Busemann functions]
Let $c:[0,\infty)\to \hs$ be a geodesic ray. The function $b_c:\hs\to\rls$ defined by
$$ b_c(x)\as \lim_{t\to\infty} \left[d\left(x,c(t) \right) -t \right],  \quad x\in \hs,$$
is called the \emph{Busemann function} associated to the ray $c,$ see \cite[Definition~II.8.17]{book}. Busemann functions are convex and $1$-Lipschitz. Concrete examples of Busemann functions are given in \cite[p.~273]{book}. Another explicit example of a Busemann function in the Hadamard space of positive definite $n\times n$ matrices with real entries can be found in~\cite[Proposition~10.69]{book}. The sublevel sets of Busemann functions are called \emph{horoballs} and carry a lot of information about the geometry of the space in question, see \cite{book} and the references therein.
\end{exa}

\begin{exa}[Energy functional]
The energy functional is another important instance of a convex function on an Hadamard space, see \cite[Chapter~7]{jost1}, or more generally in \cite[Chapter~4]{jost2}. Indeed, the energy functional is convex and lsc on a suitable Hadamard space of $\mathcal{L}^2$-mappings. Minimizers of the energy functional are called \emph{harmonic maps,} and are important in both geometry and analysis. For a probabilistic approach to harmonic maps in Hadamard spaces, see~\cite{sturm-markov1,sturm-markov2,sturm-semigr}.
\end{exa}

Having now a convex lsc function $f:\hs\to\exrls,$ we consider the minimization problem
\begin{equation} \label{eq:problem}
\text{find } x\in \hs \text{ such that } f(x)=\inf_\hs f,
\end{equation}
whose importance is with regard to the above examples more than obvious. This problem has been in the context of Hadamard spaces studied in~\cite{ppa,jost95,jost2,jost-ch,jost1,mayer}, and more specifically on manifolds of nonpositive sectional curvature in~\cite{sevilla,sevilla2}. To study the minimization problem~\eqref{eq:problem} we use the following theory.

\subsection{Nonlinear resolvents and semigroups}
Let us first recall some definitions. For $\lambda>0,$ define the \emph{Moreau-Yosida envelope} of~$f$ as
\begin{align}\label{defappr}
f_\lambda(x) & \as\min_{y\in \hs} \left[f(y)+\frac1{2\lambda}d(x,y)^2\right],\quad x\in \hs, \\
\intertext{and the \emph{resolvent} of $f$ as} \label{eq:defres}
J_\lambda(x) & \as\argmin_{y\in \hs} \left[f(y)+\frac1{2\lambda}d(x,y)^2\right],\quad x\in \hs,
\end{align}
and put $J_0(x)\as x,$ for all $x\in \hs.$ This definition in metric spaces with no linear structure first appeared in \cite{jost95}. The mapping $J_\lambda$ is well defined for all $\lambda\geq0,$ see \cite[Lemma~2]{jost95} and \cite[Theorem~1.8]{mayer}.

Finally, the \emph{(gradient flow) semigroup} $\left(S_t\right)_{t\geq0}$ of $f$ is given as
\begin{equation} \label{eq:defsem}
S_t (x)\as \lim_{n\to\infty} \left(J_{\frac{t}n}\right)^{(n)} (x),\quad x\in \cldom f.
\end{equation}
Here $\cldom f$ denotes the closure of $\dom f\as\left\{x\in\hs:f(x)<\infty\right\}.$ The limit in~\eqref{eq:defsem} is uniform with respect to $t$ on bounded subintervals of $[0,\infty),$ and $\left(S_t\right)_{t\geq0}$ is a strongly continuous semigroup of nonexpansive mappings, see \cite[Theorem~1.3.13]{jost-ch}, and \cite[Theorem~1.13]{mayer}. Surprisingly, these notions can be extended into a much more general setting~\cite{ambrosio}. On the other hand if $\hs$ is a Hilbert space, $u_0\in\hs,$ and we put $u(t)=S_t\left(u_0\right),$ for $t\in[0,\infty),$ we obtain a ``classical'' solution to the parabolic problem~\eqref{eq:gradflowproblem} with the initial condition $u(0)=u_0.$

When establishing the existence of harmonic maps, Jost~\cite[Theorem 1]{jost95} proved the following result. For the details, see also \cite{jost2,jost-ch,jost1}.
\begin{thm}\label{thm:jost}
Let $(\hs,d)$ be an Hadamard space, $f:\hs\to\exrls$ be a~convex lsc function, and $x_0\in \hs.$ Assume there exists a sequence $\left(\lambda_n\right)\subset(0,\infty)$ with $\lambda_n\to\infty$ such that $\left(J_{\lambda_n}x_0\right)$ is a bounded sequence. Then $J_\lambda x_0$  converges to a minimizer of the function~$f,$ as~$\lambda\to\infty.$ In particular, the function $f$ attains its minimum.
\end{thm}

In spite of the significance of the convergence of $J_\lambda,$ as $\lambda\to\infty,$ it is more desirable to establish convergence of the semigroup $S_t$ as $t\to\infty.$ Unfortunately, the semigroup $\left(S_t\right)$ convergences only weakly~\cite{ppa}.
\begin{thm}\cite[Theorem~1.5]{ppa} \label{thm:flow}
Let $(\hs,d)$ be an Hadamard space, and $f:\hs\to\exrls$ be lsc convex. Assume that $f$ attains its minimum on $\hs.$ Then, given a~starting point $x\in\cldom f,$ the gradient flow $S_t x$ weakly converges to a minimizer of $f,$ as $t\to\infty.$
\end{thm}

The \emph{proximal point algorithm} is a discrete version of the gradient flow. It starts at a point $x_0\in \hs,$ and generates at the $n$-th step, $n\in\nat,$ the point 
\begin{equation} \label{eq:ppa}
 x_n=\argmin_{y\in \hs}\left[f(y)+\frac1{2\lambda_n}d\left(y,x_{n-1}\right)^2\right].
\end{equation}
where $\lambda_n>0$ for each $n\in\nat.$ As one would expect, the asymptotic behavior of the proximal point algorithm is the same as in the case of the flow.
\begin{thm}\cite[Theorem~1.4]{ppa} \label{thm:ppa}
Let $(\hs,d)$ be an Hadamard space, and $f:\hs\to\exrls$ be a~convex lsc function. Assume that $f$ attains its minimum on $\hs.$ Then, for an arbitrary starting point $x_0\in \hs,$ and a sequence of positive reals $\left(\lambda_n\right)$ such that $\sum_1^\infty\lambda_n=\infty,$ the sequence $(x_n)\subset \hs$ defined by \eqref{eq:ppa} weakly converges to a~minimizer of $f.$
\end{thm}

It is known that the convergence in Theorems~\ref{thm:flow} and~\ref{thm:ppa} is not in general strong. The counterexamples are however known only in Hilbert spaces~\cite{ppa}, and we may in parallel with~\cite{reich} raise a question, whether there exists a lsc convex function on the Hilbert ball such that the proximal point algorithm, or the gradient flow, does not converge strongly, see also~\cite[Remark~9.4]{kopecka}.

\subsection{The Mosco convergence} Theorems~\ref{thm:flow} and~\ref{thm:ppa} established asymptotic behavior of the resolvent and semigroup of a \emph{fixed} convex lsc function on an Hadamard space. See also~\cite{jost-ch,mayer}. In the present paper we however study the convergence of resolvents and semigroups with respect to the convergence of a \emph{sequence} of generating functions. This extends the celebrated results of Attouch~\cite[Th\'eor\`eme 1.2]{attouch79}\footnote{see also~\cite[Theorem 3.26]{attouch-b}}, and B\'enilan, Brezis and Pazy~\cite[Th\'eor\`eme~4.2]{brezis-b} into spaces with no linear structure. Related Mosco's results for quadratic forms on Hilbert spaces can be found in his seminal paper~\cite{mosco}.

One of the most important approaches in variational analysis is to consider a~\emph{sequence} of (convex lsc) functions converging in a certain sense to a (convex lsc) function in question, and study the relationship between their respective minimizers. Among numerous types of such convergences we choose the $\Gamma$-convergence, and mainly the Mosco convergence for their significance in analysis~\cite{attouch-b,maso}. These types of convergences were first studied in the context of metric spaces of nonpositive curvature by Jost in~\cite{jost-ch} and enabled him to define the energy functional. See also~\cite{jost94,jost95,jost2,jost1,japan}. 

One of our main results (Theorem~\ref{thm:mosco}) states that the Mosco convergence of a~sequence of convex lsc functions implies the convergence of their resolvents. This property was actually used by Jost as the definition of the Mosco convergence~\cite[Definition~3.2]{jost-ch}. We also remark that as far as \emph{nonnegative} convex lsc functions are concerned, our Theorem~\ref{thm:mosco} is contained in~\cite[Proposition~5.2]{japan}.
\begin{thmnn}[Theorem~\ref{thm:mosco} below] Let $(\hs,d)$ be an Hadamard space, $f:\hs\to\exrls$ and $f^n:\hs\to\exrls$ be convex lsc functions, for $n\in\nat.$ Let $f_\lambda,$ and $f_\lambda^n$ be the corresponding Moreau-Yosida envelopes, with $\lambda>0,$ and let $J_\lambda,$ and $J_\lambda^n$ be the corresponding resolvents, with $\lambda>0.$ If $f^n\to f$ in the sense of Mosco, as $n\to\infty,$ then
\begin{align}
 \lim_{n\to\infty} f_\lambda^n(x)= f_\lambda(x), \\
\intertext{and,}
 \lim_{n\to\infty} J_\lambda^n(x)= J_\lambda(x).  
\end{align}
for any $\lambda>0,$ and $x\in \hs.$
\end{thmnn}
As an application, we obtain that the Mosco convergence of convex closed sets implies the Frol\'ik-Wijsman convergence; see Corollary~\ref{cor:wijsman}.

We further show (Theorem~\ref{thm:semigr}) that the convergence of resolvents gives the convergence of the gradient flow semigroups.
\begin{thmnn}[Theorem~\ref{thm:semigr} below]
Let $(\hs,d)$ be an Hadamard space. Assume $f:\hs\to\exrls$ and $f^n:\hs\to\exrls,$ for $n\in\nat,$ are lsc convex functions. Let $J_\lambda,$ and $J_\lambda^n$ be the corresponding resolvents, with $\lambda>0,$ and let $S_t$ and $S_t^n$ be the corresponding semigroups, with $t>0.$ Assume that for any $x\in \cldom f$ and $\lambda>0,$ we have 
$$\lim_{n\to\infty} J_\lambda^n x = J_\lambda x.$$
Then
$$\lim_{n\to\infty} S_t^n x= S_t x,$$
for any $x\in \cldom f,$ and $t>0.$ 
\end{thmnn}

We also prove these results when instead of having one \emph{fixed} Hadamard space, we consider a \emph{sequence} of such spaces equipped with a so-called asymptotic relation; see Section~\ref{subsec:asy} for the definition. The concept of an asymptotic relation was introduced in~\cite{japan}, and extends Gromov-Hausdorff convergence beyond local compactness; see~\cite{japan} and the references therein, especially~\cite{fukaya,gromov}.

We note that the currently best result on semigroup convergence on asymptotic relations is limited to Hilbert spaces and quadratic forms \cite[Theorem~5.27]{japan}. It is also worth mentioning that the convergence of operators on \emph{varying} spaces is a~rather new and unexplored topic \cite{koles,japan2,japan,nittka}.

\subsection{Paper organization} Let us briefly outline the remainder of the paper.
\begin{itemize}
 \item Section \ref{sec:pre}: Notation and preliminary facts are established.
 \item Section \ref{sec:weak}: We collect and unify facts on the weak convergence in Hadamard spaces.
 \item Section \ref{sec:main}: We show that the Mosco convergence of a sequence of functions implies strong convergence of the resolvents and semigroups, and that the same holds on a \emph{sequence} of Hadamard spaces.
\end{itemize}

Although the concepts appearing in the present paper (like weak convergence, or asymptotic relation) make sense for \emph{nets,} that is, sequences indexed by an arbitrary directed set, we choose to work with ordinary \emph{sequences} for the sake of simplicity. 

\subsection*{Acknowledgments.}
 I am grateful to Martin Kell for his very valuable comments. I would like to thank the referee for reading the manuscript with an extraordinary care and pointing out many misprints and inaccuracies.

\section{Preliminaries} \label{sec:pre}
We first recall basic notation and facts concerning Hadamard spaces. For further details on the subject, the reader is referred to \cite{book}. Let $(\hs,d)$ be an Hadamard space. Having two points $x,y\in \hs,$ we denote the geodesic segment from $x$ to $y$ by $[x,y].$ We usually do not distinguish between a geodesic and its geodesic segment, as no confusion can arise. For a point $z\in[x,y],$ we write $z=(1-t)x+ty,$ where $t=d(x,z)/d(x,y).$ 

Given $x,y,z\in \hs,$ the symbol $\alpha(y,x,z)$ denotes the (Alexandrov) angle between the geodesics $[x,y]$ and $[x,z]$.

For a function $f:\hs\to\exrls$ we denote $\dom f=\left\{x\in \hs:f(x)<\infty\right\}.$ If $\dom f\neq\emptyset,$ we say $f$ is \emph{proper.} To avoid trivial situations we often assume this property without explicit mentioning. A point $x\in\hs$ is called a \emph{minimizer} of $f$ if $f(x)=\inf_\hs f.$

If $F:\hs\to \hs$ is a mapping, we denote its $k^{\mathrm{th}}$ power, with $k\in\nat,$ by
$$F^{(k)}x\as\left(F\circ\dots\circ F\right) x,\quad x\in \hs,$$
where $F$ appears $k$-times on the right hand side.
Lower semicontinuity is abbreviated as lsc. 

\subsection{Convex sets and functions on Hadamard spaces} \label{subsec:prelim} Recall that a set $C\subset \hs$ is \emph{convex} if $x,y\in C$ implies $[x,y]\subset C.$ A~function $f:\hs\to\exrls$ is \emph{convex} provided $f\circ\gamma:[0,1]\to\exrls$ is convex for any geodesic $\gamma:[0,1]\to\hs.$ Note that the distance function $d_C$ is convex and continuous, see Example~\ref{exa:dist}.
\begin{prop}
Let $(\hs,d)$ be an Hadamard space and $C\subset \hs$ be complete and convex. Then:
\begin{enumerate}
\item For every $x\in \hs,$ there exists a unique point $P_C(x)\in C$ such that
$$d\left(x,P_C(x)\right)=d_C(x).$$
\item If $y\in\left[x,P_C(x)\right],$ then $P_C(x)=P_C(y).$
\item If $x\in \hs\setminus C$ and $y\in C$ such that $P_C(x)\neq y,$ then $\alpha\left(x,P_C(x),y\right)\geq\frac\pi2.$
\item The mapping $P_C:\hs\to C$ is nonexpansive and is called the \emph{metric projection} onto $C.$
\end{enumerate}
\label{prop:proj}
\end{prop}
\begin{proof} See \cite[Proposition~2.4, p.176]{book}. \end{proof}

A function $h:\hs\to\exrls$ is \emph{strongly convex} with parameter $\beta>0$ if,
$$ h\left((1-t)x+ty\right)\leq (1-t) h(x)+t h(y)-\beta t(1-t)d(x,y)^2,$$
for any $x,y\in \hs$ and any $t\in[0,1].$ 

\begin{rem}\label{rem:stronglyconvex}
Having established this terminology, the inequality \eqref{eq:cat} says that, for a fixed $x_0\in \hs,$ the function $d(\cdot,x_0)^2$ is strongly convex with parameter~$\beta=1.$
\end{rem}

\begin{lem}\label{lem:limsup}
Let $(\hs,d)$ be an Hadamard space, and $f_n:\hs\to\exrls$ be functions, for $n\in\nat.$ Assume that all $f_n$ are strongly convex with a common parameter $\beta.$ Then the function
$$f=\limsup_{n\to\infty} f_n$$
is strongly convex with parameter $\beta.$
\end{lem}
\begin{proof} The proof of Lemma~\ref{lem:limsup} is rather straightforward, and is therefore omitted.\end{proof}

Strongly convex functions have the following nice property. 
\begin{lem}\label{lem:minofuc}
Let $(\hs,d)$ be an Hadamard space, and $f:\hs\to\exrls$ be a~strongly convex lsc function. Then there exists a unique minimizer of~$f.$
\end{lem}
\begin{proof}
Uniqueness is clear. Existence was established in~\cite[Lemma~1.7]{mayer}.
\end{proof}

Let $f:\hs\to\exrls$ be a convex lsc function, and $x\in\dom f.$ Define the \emph{slope} of $f$ by
$$|\partial f|(x)\as\limsup_{y\to x}\frac{\max\left\{f(x)-f(y),0\right\}}{d(x,y)},$$
and $\dom |\partial f|\as\left\{x\in \hs:|\partial f|(x)<\infty \right\}.$ We set $|\partial f|(x)\as\infty$ when $x\notin \dom f.$ See~\cite[Definition~1.2.4]{ambrosio}.
\begin{prop} \label{prop:slope}
Let $(\hs,d)$ be an Hadamard space, and $f:\hs\to\exrls$ be a convex lsc function. Then for any $x\in \hs,$ and $\lambda>0,$ we have $J_\lambda x\in\dom |\partial f|,$ and
\begin{equation} \label{eq:slope}
 |\partial f|\left(J_\lambda x\right) \leq\frac{d\left(x,J_\lambda x\right)}\lambda
\end{equation}
holds. In particular, we have $\dom |\partial f|$ is dense in $\dom f.$ If, moreover $x\in\dom |\partial f|,$ we have
\begin{equation} \label{eq:slope2}
 |\partial f|\left(J_\lambda x\right)^2 \leq\frac{d\left(x,J_\lambda x\right)^2}{\lambda^2}\leq 2\frac{f(x)-f_\lambda(x)}\lambda\leq  |\partial f|(x)^2
\end{equation}
\end{prop}
\begin{proof} See \cite[Lemma~3.1.3]{ambrosio}, and \cite[Theorem~3.1.6]{ambrosio}. \end{proof}

\begin{prop}
 Let $(\hs,d)$ be an Hadamard space, and $f:\hs\to\exrls$ be convex lsc. For $x\in\dom f,$ we have the error estimate
\begin{equation}
 \label{eq:error}
d\left(S_tx,\left(J_{\frac{t}{n}}\right)^{(n)}x\right)\leq \frac{t}{\sqrt2n}|\partial f|(x),
\end{equation}
for any $t>0,$ and $n\in\nat.$
\end{prop}
\begin{proof} See \cite[Theorem~4.0.4]{ambrosio}. \end{proof}

\subsection{Asymptotic relations} \label{subsec:asy} Let $(X,d)$ be a metric space and $\left(X^n,d^n\right)_{n\in\nat}$ be a~sequence of metric spaces. Denote the disjoint union
$$\calx=\left(\bigsqcup_{n\in\nat}X^n \right) \sqcup X .$$
Following \cite[Definition~3.1]{japan}, we call a topology on $\calx$ an \emph{asymptotic relation} between $\left(X^n\right)$ and $X$ if the following conditions are satisfied:
\begin{enumerate}
 \item[(A1)] All $X^n,$ with $n\in\nat,$ and $X$ are closed in $\calx,$ and the restricted topology of $\calx$ on each of $X^n$ and $X$ coincides with its original topology.
 \item[(A2)] For any $x\in X$ there exists a sequence $x^n\in X^n$ converging to $x$ in $\calx.$
 \item[(A3)] If a sequence $x^n\in X^n$ converges to $x\in X$ in $\calx,$ and a sequence $y^n\in X^n$ converges to $y\in X$ in $\calx,$ we have 
$$d^n\left(x^n,y^n\right)\to d(x,y).$$
 \item[(A4)] If a sequence $x^n\in X^n$ converges to $x\in X$ in $\calx,$ and a sequence $y^n\in X^n$ is such that $d^n\left(x^n,y^n\right)\to 0,$ then $\left(y^n\right)$ converges to $x$ in $\calx.$
\end{enumerate}
As already mentioned in the Introduction, the Gromov-Hausdorff convergence is an instance of an asymptotic relation~\cite{japan}. The convergence of a sequence $x^n\in X^n$ to a point $x\in X$ will be denoted classically $x^n\to x.$

A sequence $x^n\in X^n$ is \emph{bounded} if there exists a convergent sequence $y^n\in X^n$ such that $d^n\left(x^n,y^n\right)$ is bounded.

Having an asymptotic relation over a disjoint union $\bigsqcup_{n\in\nat}\hs^n \sqcup \hs$ of Hadamard spaces $\hs^n,\hs,$ we say that a sequence of geodesics $\gamma^n:[0,1]\to \hs^n$ converges to a geodesic $\gamma:[0,1]\to \hs$ if $\gamma^n(0)\to\gamma(0),$ and $\gamma^n(1)\to\gamma(1).$
We will need the following fact from \cite[Proposition~5.1]{japan}.
\begin{prop}\label{prop:asy}
Let $(\hs,d)$ be an Hadamard space and $\left(\hs^n,d^n\right)_{n\in\nat}$ be a sequence of Hadamard spaces, and suppose an asymptotic relation is given. If geodesics $\gamma^n:[0,1]\to \hs^n,$ converge to a geodesic $\gamma:[0,1]\to \hs,$ then $\gamma^n(t)\to\gamma(t),$ for any $t\in[0,1].$
\end{prop}

When no confusion is likely, we denote the metric on $\hs^n$ by $d$ instead of $d^n.$

\section{Weak convergence theory} \label{sec:weak}

This section is to give a systematic account on the weak convergence in Hadamard spaces. The theory has been hitherto only scattered in the literature, and it seems highly desirable to collect and unify all of that at one place. This could, among other things, help to avoid future rediscoveries; see Remark~\ref{rem:history}.
Also, the weak convergence will play a key role in this paper.

From our perspective, the importance of the weak convergence comes from the fact that, like in Hilbert spaces, 
\begin{itemize}
 \item a bounded sequence has a weakly convergent subsequence,
 \item a convex closed set is (sequentially) weakly closed, and
 \item a convex lsc function is (sequentially) weakly lsc.
\end{itemize}
These features are essential not only in the present paper, but also in \cite{ppa,apm,efl,jost94,kp}, and many others.

Most of this section is devoted to the weak convergence on a \emph{fixed} Hadamard space. We bring the relevant results (including the proofs) in detail. In the end we deal with the weak convergence on a \emph{sequence} of such spaces equipped with an asymptotic relation, which was defined in Section~\ref{subsec:asy}.

As already mentioned in the Introduction, we will work with sequences. Though it causes no troubles to replace sequences by nets.

\subsection{Basic theory of the weak convergence}
Let $(\hs,d)$ be an Hadamard space and $(x_n)\subset \hs$ be a bounded sequence. Define the function $\omega:\hs\to[0,\infty)$ as
\begin{equation}\label{eq:omega}
\omega\left(x\left(x_n\right)\right)\as\limsup_{n\to\infty} d(x,x_n)^2,\quad x\in \hs.
\end{equation}
Unlike \cite{dks} we use the square power, which immediately yields a unique minimizer of $\omega.$ Indeed, we know from Example~\ref{exa:limsup} that $\omega$ is strongly convex, and locally Lipschitz. Via Lemma~\ref{lem:minofuc} we then obtain the following.
\begin{lem} \label{lem:asycenter}
 Let $(x_n)\subset \hs$ be a bounded sequence. Then function $\omega$ defined in~\eqref{eq:omega} has a unique minimizer, which we call the asymptotic center of~$(x_n).$
\end{lem}
We shall say that $(x_n)\subset \hs$ \emph{weakly converges} to a point $x\in \hs$ if $x$ is the asymptotic center of each subsequence of $(x_n).$ We use the notation $x_n\wto x.$ Clearly, if $x_n\to x,$ then $x_n\wto x.$

If there is a subsequence $(x_{n_k})$ of $(x_n)$ such that $x_{n_k}\wto z$ for some $z\in \hs,$ we say that $z$ is a \emph{weak cluster point} of the sequence $(x_n).$ 

It is convenient to denote $r(x_n)\as\inf_{x\in \hs}\omega\left(x,\left(x_n\right)\right).$

\begin{prop} \label{prop:weakcluster}
Each bounded sequence has a weakly convergent subsequence, or in other words, each bounded sequence has a weak cluster point.
\end{prop}
\begin{proof} We mimic the proof of \cite[Lemma~15.2]{goebelkirk}. If $(u_n)$ is a subsequence of $(v_n),$ we will use the notation $(u_n)\prec(v_n).$ Let $(x_n)$ be a bounded sequence. Denote
$$\rho_0=\inf\left\{r(v_n):(v_n)\prec(x_n)\right\},$$
and select $\left(v_n^1\right)\prec(x_n)$ such that
$$r\left(v_n^1\right)<\rho_0+1.$$
Denote
$$\rho_1=\inf\left\{r(v_n):(v_n)\prec\left(v_n^1\right)\right\}.$$
Having $\left(v_n^i\right)\prec\left(v_n^{i-1}\right),$ set
$$\rho_i=\inf\left\{r(v_n):(v_n)\prec\left(v_n^i\right)\right\}.$$
Select $\left(v_n^{i+1}\right)\prec\left(v_n^i\right)$ such that
$$r\left(v_n^{i+1}\right)\leq\rho_i+\frac1{i+1}.$$
Since $(\rho_n)$ is non-decreasing and bounded from above by $r(x_n),$ it has a limit, say~$\rho.$

Now take the diagonal sequence $\left(v_k^k\right),$ and fix $i\in\nat.$ Then $\left(v_k^k\right)$ is a subsequence (modulo the first $i-1$ elements) of $\left(v_n^i\right),$ and hence $r\left(v_k^k\right)\geq \rho_i,$  On the other hand, for the same fixed $i\in\nat,$ we have that $\left(v_k^k\right)$ is a subsequence (modulo the first $i$~elements) of $\left(v_n^{i+1}\right),$ which gives $r\left(v_k^k\right)\leq \rho_i+1/(i+1).$ Taking the limit $i\to\infty$ gives $r\left(v_k^k\right)=\rho.$

Since any subsequence $(u_n)$ of $\left(v_k^k\right)$ also (for the same reasons) satisfies the inequalities $r\left(v_k^k\right)\geq \rho_i,$ and $r\left(v_k^k\right)\leq \rho_i+1/(i+1),$ for any $i\in\nat,$ one gets
\begin{equation}\label{eq:rho}
r(u_n)=\rho.
\end{equation}

We can conclude that $\left(v_k^k\right)$ is the desired subsequence. Lemma~\ref{lem:asycenter} yields a unique point $x\in \hs$ such that $\limsup_{k\to\infty} d\left(x,v_k^k\right)^2=\rho.$ By~\eqref{eq:rho} we get $v_k^k\wto x.$
\end{proof}
Jost \cite[Theorem~2.1]{jost94} gave a different proof of Proposition~\ref{prop:weakcluster}, also based on a~diagonalization argument.

The following useful characterization comes from~\cite[Proposition~5.2]{efl}.
\begin{prop} \label{prop:rafa}
For a bounded sequence $(x_n)\subset \hs,$ the following assertions are equivalent:
\begin{enumerate}
 \item weakly converges to a point $x\in \hs,$                                    \label{i:rafa:i}
 \item for any geodesic $\gamma:[0,1]\to\hs$ with $x\in\gamma,$ we have                     \label{i:rafa:ii}
$$d\left(x,P_\gamma(x_n)\right)\to 0, \quad \text{as } n\to\infty,$$
 \item for any $y\in \hs,$ we have                                                \label{i:rafa:iii}
$$d\left(x,P_{[x,y]}(x_n)\right)\to 0, \quad \text{as } n\to\infty.$$
\end{enumerate}
\end{prop}
\begin{proof} 
\eqref{i:rafa:i}$\implies$\eqref{i:rafa:ii}: Let $\gamma:[0,1]\to\hs$ be a geodesic with $x\in\gamma.$ If
$$\limsup_{n\to\infty} d\left(x,P_\gamma(x_n)\right)>0,$$
then there exists a subsequence $(y_n)$ of $(x_n)$ such that $P_\gamma(y_n)$ converges to some $y\in\gamma\setminus\{x\}.$ But then
$$\limsup_{n\to\infty} d(y,y_n)^2\leq \limsup_{n\to\infty} d(x,y_n)^2,$$
which contradicts $y_n\wto x.$

\eqref{i:rafa:ii}$\implies$\eqref{i:rafa:iii}: Trivial.

\eqref{i:rafa:iii}$\implies$\eqref{i:rafa:i}: If the sequence $(x_n)$ does not converge weakly to $x,$ then there exists a subsequence $(y_n)$ of $(x_n)$ such that
$$\limsup_{n\to\infty} d(y,y_n)^2< \limsup_{n\to\infty} d(x,y_n)^2,$$
for some $y\in \hs\setminus\{x\}.$ But then
$$\limsup_{n\to\infty}d\left(x,P_{[x,y]}(y_n)\right)>0,$$
which contradicts \eqref{i:rafa:iii}.
\end{proof}

One can easily see from Proposition~\ref{prop:rafa} that in Hilbert spaces, the notion of weak convergence defined above coincides with the classical weak convergence.

In the following series of lemmas we extend various properties of the weak convergence in Hilbert spaces to Hadamard spaces. 

The property in Lemma~\ref{lem:opial} comes from \cite{opial} where the author proves it for Hilbert spaces.
\begin{lem}\label{lem:opial}
Any Hadamard space $(\hs,d)$ enjoys the Opial property, that is, for any sequence $(x_n)\subset \hs$ weakly converging to a point $x\in \hs$ we have
$$ \liminf_{n\to\infty} d(x_n,x)<\liminf_{n\to\infty} d(x_n,z)$$
for any $z\in \hs\setminus\{x\}.$ 
\end{lem}
\begin{proof}
 Follows from Proposition~\ref{prop:rafa}.
\end{proof}

\begin{lem} \label{lem:wtos}
 Let $(x_n)\subset \hs,$ and $x\in \hs.$ Then $x_n\to x$ if and only if $x_n\wto x,$ and $d(x_n,y)\to d(x,y)$ for some $y\in \hs.$
\end{lem}
\begin{proof}
 Let $x_n\wto x,$ and $d(x_n,y)\to d(x,y)$ for some $y\in \hs.$ Note that $d\left(x_n,y\right)^2\geq d\left(x_n,P_{[x,y]}x_n\right)^2+ d\left(P_{[x,y]}x_n,y\right)^2$ holds. Since $P_{[x,y]}x_n$ converges to $x,$ for any $\eps>0,$ we have
$$\eps+d\left(x_n,y\right)^2\geq d\left(x_n,x\right)^2+ d\left(x,y\right)^2,$$
for all sufficiently high $n\in\nat.$ Hence
$$\eps\geq \limsup_{n\to\infty} d\left(x_n,x\right)^2.$$
The converse implication is trivial.
\end{proof}

A somewhat quantified version of Lemma~\ref{lem:wtos} appeared in~\cite[Theorem~3.9]{kp}. We slightly weaken the assumptions and give a simpler proof here.
\begin{prop}
Let $y\in \hs.$ Then for any $\eps>0$ there exists $\delta>0$ such that for every sequence $(x_n)\subset \hs$ weakly converging to a point $x\in \hs,$ and satisfying $d(y,x_n)\leq1,$ and $\limsup_{m,n\to\infty} d(x_m,x_n)>\eps,$ we have $d(x,y)\leq1-\delta.$
\end{prop}
\begin{proof}
Select a subsequence of $(x_n),$ still denoted $(x_n),$ such that $d(x,x_n)>\eps/2,$ for all $n\in\nat.$ Choose $n_0\in\nat$ so that for all $n>n_0,$ we have
\begin{align*}
d(y,x_n)^2 & \geq d(x,y)^2+d(x,x_n)^2 -\frac{\eps^2}8.\\
\intertext{Then,}
d(x,y)^2 & \leq 1-\frac{\eps^2}4+\frac{\eps^2}8=1-\frac{\eps^2}8.
\end{align*}
\end{proof}

The conclusion of~\cite[Theorem~3.9]{kp} was called the Kadec-Klee property. Recall that a Banach space has the Kadec-Klee property if the weak and norm topologies coincide on the unit sphere~\cite{cz}. Some authors give this name to a weaker property: if the weak and strong convergence of sequences coincide on the unit sphere in a~Banach space. 

Lemmas~\ref{lem:wclosure} and~\ref{lem:convexlsc} appeared in~\cite{apm}.
\begin{lem}
Let $C\subset \hs$ a closed convex set. If $(x_n)\subset C$ and $x_n\wto x\in \hs,$ then $x\in C.$ \label{lem:wclosure}
\end{lem}
\begin{proof}
Assume that $x\notin C$ and denote $\gamma=\left[x,P_C(x)\right].$ We claim that $P_\gamma(x_n)=P_C(x)$ for all $n\in\nat.$ Indeed, if for some $m\in\nat$ we had $P_\gamma(x_m)\neq P_C(x),$ then by Proposition \ref{prop:proj}, we would have both
$$\alpha\left(x_m,P_C(x),P_\gamma(x_m) \right) \geq \frac\pi2, \quad \alpha\left(x_m,P_\gamma(x_m),P_C(x) \right) \geq \frac\pi2,$$
which is impossible. Finally,
$$ d(P_\gamma(x_n),x)=d\left(P_C(x),x\right)\nrightarrow 0, \quad n\to\infty,$$
which, by Proposition \ref{prop:rafa}, contradicts $x_n\wto x.$
\end{proof}

\begin{defi}
We shall say that a function $f:\hs\to\exrls$ is \emph{weakly lsc} at a~given point $x\in\dom f$ if
$$\liminf_{n\to\infty} f(x_n)\geq f(x),$$
for each sequence $x_n \wto x.$ We say that $f$ is weakly lsc if it is lsc at any $x\in\dom f.$
\end{defi}

Given a set $A\subset\hs,$ its \emph{closed convex hull} is defined by
$$\clco A\as \bigcap\left\{C\subset\hs:C \text{ closed convex, } A\subset C\right\}.$$ 
\begin{lem} \label{lem:convexlsc}
If $f:\hs\to\exrls$ a lsc convex function, then it is weakly lsc.
\end{lem}
\begin{proof}
By contradiction. Let $(x_n)\subset \hs, x\in\dom f$ and $x_n\wto x.$ Suppose that
$$\liminf_{n\to\infty} f(x_n) < f(x).$$
That is, there exist a subsequence $(x_{n_k}),$ index $k_0\in\nat,$ and $\delta>0$ such that $ f(x_{n_k}) < f(x)-\delta$
for all $k>k_0.$ By lower semicontinuity and convexity of $f,$ we get
$$ f(y) \leq f(x)-\delta$$
for all $y\in \clco \{x_{n_k}:k>k_0\}.$ But this, through Lemma \ref{lem:wclosure}, yields a contradiction to $x_n\wto x.$
\end{proof}

\begin{cor} Let $C\subset \hs$ a closed convex set. The distance function $d_C$ as well as its square $d_C^2$ are weakly lsc.
\label{cor:distlsc}
\end{cor}

We now turn our attention to Fej\'er monotonicity. Combettes~\cite{combettes} surveys the importance and usefulness of this feature in optimization, and also describes its history. In the context of Hadamard space, this property was first used in~\cite{apm}.

A sequence $(x_n)\subset \hs$ is \emph{Fej\'er monotone} with respect to a set $C\subset \hs$ if, for any $c\in C,$ we have
$$ d(x_{n+1},c) \leq d(x_n,c), \quad n\in\nat.$$ The following proposition comes from~\cite[Proposition~3.3]{apm}.

\begin{prop}
Let $C\subset \hs$ be a closed convex set. Assume $(x_n)\subset \hs$ is a~Fej\'er monotone sequence with respect to $C.$ Then we have:
\begin{enumerate}
\item $(x_n)$ is bounded, \label{item:i}
\item $d_C(x_{n+1}) \leq d_C(x_n)$ for each $n\in\nat,$ \label{item:ii}
\item $(x_n)$ weakly converges to some $x\in C$ if and only if all weak cluster points of $(x_n)$ belong to $C,$ \label{item:iii}
\item $(x_n)$ converges to some $x\in C$ if and only if $d(x_n,C)\to 0.$ \label{item:iv}
\end{enumerate}
\label{prop:fejer}
\end{prop}
\begin{proof}
\eqref{item:i} and \eqref{item:ii} are easy. Let us prove the nontrivial implication of \eqref{item:iii}. Assume that all weak cluster points of $(x_n)$ lie in $C.$ It suffices to show that $(x_n)$ has a~unique cluster point. By contradiction, let $c_1,c_2\in C,$ with $c_1\neq c_2,$ be weak cluster points of $(x_n).$  That is, there are subsequences $(x_{n_k})$ and $(x_{m_k})$ such that $x_{n_k}\wto c_1$ and $x_{m_k} \wto c_2.$ Without loss of generality, assume $r(x_{n_k})\leq r(x_{m_k}).$ For any $\eps>0$ there exists $k_0\in\nat$ such that $ d(x_{n_k},c_1)^2<r(x_{n_k})+\eps,$ for all $k\geq k_0.$ By Fej\'er monotonicity we also have $ d(x_{m_k},c_1)^2<r(x_{n_k})+\eps,$ for all $m_k\geq n_{k_0}.$ Hence, there exists $k_1\in\nat$ such that $ d(x_{m_k},c_1)^2<r(x_{m_k})+\eps,$ for all $k\geq k_1.$ But this contradicts the fact that $c_2$ is the unique asymptotic center of $(x_{m_k}).$

Now we prove \eqref{item:iv}. Suppose $d(x_n,C)\to 0.$ Since for all $k\in\nat$ we have
\begin{subequations}
\begin{align}
d(x_{n+k},x_n) &\leq d\left(x_{n+k},P_C (x_n)\right)+d\left(x_n,P_C (x_n)\right) \\
\intertext{and hence, by Fej\'er monotonicity,}
d(x_{n+k},x_n) &\leq d\left(x_n,P_C (x_n)\right)+d\left(x_n,P_C (x_n)\right) \leq 2d(x_n,C),
\end{align}
\label{eq:fejer}
\end{subequations}
which gives that $(x_n)$ is Cauchy and therefore converges to a point from $C.$ The converse implication in \eqref{item:iv} is trivial.
\end{proof}

\begin{rem}[Weak topology] Having the notion of the weak convergence, it is now natural to ask whether there is a topology which generates this convergence, see for instance \cite[p.~3696]{kp}. We propose the following definition. Let $(\hs,d)$ be an Hadamard space. We will say that a set $M\subset \hs$ is \emph{weakly open} if, for each $x_0\in M,$ there is $\eps>0$ and a finite family of nontrivial geodesics $\gamma_1,\dots,\gamma_N$ containing $x_0$ such that the set
\begin{equation} \label{neigh}
U_{x_0}(\eps,\gamma_1,\dots,\gamma_N)\as\left\{x\in \hs: d\left(x_0,P_{\gamma_i}x\right)<\eps, \:i=1,\dots,N   \right\}
\end{equation}
is contained in $M.$ The collection of all weakly open sets in $\hs$ will be denoted $\tau_w.$ It is easy to see that $\tau_w$ is a Hausdorff topology on $\hs,$ and any weakly converging (according to the above definition) sequence also converges in $\tau_w.$ We however are not able to prove the converse. The difficulty is that we do not know whether the sets in \eqref{neigh} are themselves weakly open. If $\hs$ is a Hilbert space, then $\tau_w$ of course coincides with the weak topology~$\sigma\left(\hs,\hs^*\right).$
\end{rem}

\begin{rem}[History of the weak convergence] \label{rem:history}
The notion of weak convergence in Hadamard spaces was first introduced by J\"urgen Jost in \cite[Definition~2.7]{jost94}. Sosov later defined his $\psi$- and $\phi$-convergences, both generalizing the Hilbert space weak convergence into geodesic metric spaces \cite{sos}. Then Kirk and Panyanak extended Lim's $\Delta$-convergence \cite{lim} into Hadamard spaces \cite{kp} and finally, Esp\'inola and Fern\'andez-Le\'on \cite{efl} modified Sosov's $\phi$-convergence to obtain an equivalent formulation of $\Delta$-convergence in Hadamard spaces. This is, however, exactly the original weak convergence due to Jost~\cite{jost94}.
\end{rem}

\subsection{Weak convergence on asymptotic relations}

We finish this section by describing the weak convergence on a \emph{sequence} of Hadamard spaces. This extension is due to Kuwae and Shioya~\cite{japan}.

Let $(\hs,d)$ be an Hadamard space and $\left(\hs^n,d^n\right)_{n\in\nat}$ be a sequence of Hadamard spaces, and suppose an asymptotic relation is given. We say that a~bounded sequence $x^n\in \hs^n$ \emph{weakly converges} to a point $x\in \hs$ if, for any sequence of geodesics $\gamma^n:[0,1]\to \hs^n$ converging to a geodesic $\gamma:[0,1]\to \hs$ with $\gamma(0)=x,$ we have $P_{\gamma^n}\left(x^n\right)\to x.$ This is denoted $x^n\wto x.$ It is immediate that each sequence has at most one weak limit point. Also, we have $x^n\to x$ implies $x^n\wto x.$ 

A point $c\in \hs$ is called a \emph{weak cluster point} of a sequence $\left(x^n\right)$ if there exists a~subsequence of $\left(x^n\right)$ that weakly converges to $c.$

\begin{prop} \label{prop:wclusterasy}
Let $(\hs,d)$ be an Hadamard space and $\left(\hs^n,d^n\right)_{n\in\nat}$ be a~sequence of Hadamard spaces, and suppose an asymptotic relation is given. Any bounded sequence $x^n\in \hs^n$ has a~weak cluster point.
\end{prop}
\begin{proof}
 The proof is given in \cite[Lemma~5.5]{japan}. It is based on Jost's proof of \cite[Theorem~2.1]{jost94}.
\end{proof}

The following useful fact comes from~\cite[Lemma~5.3]{japan}.
\begin{lem}\label{lem:weakar}
Let $x^n\in \hs^n,$ and $y^n\in \hs^n.$ Assume $x^n\wto x\in \hs,$ and $y^n\to y\in \hs.$ Then we have:
\begin{enumerate}
 \item $d\left(x^n,y^n\right)\leq\liminf_{n\to\infty}d\left(x^n,y^n\right),$
 \item $d\left(x^n,y^n\right)\to d(x,y)\text{ if and only if } x^n\to x.$
\end{enumerate}
\end{lem}

Note that unlike the definition of the weak convergence in \cite[Definition~5.2]{japan}, we consider only \emph{bounded} weakly converging sequences.

\section{Main results: Consequences of Mosco convergence} \label{sec:main}

In this section we show that the Mosco convergence of functions implies the convergence of the corresponding resolvents, and that the convergence of resolvents implies the convergence of the semigroups. The section has two parts: we first prove the alluded results of a \emph{fixed} Hadamard space, and then, in the second part, extend these results onto a \emph{sequence} of Hadamard spaces.

\subsection{Convergence results on a fixed space}
Let $(\hs,d)$ be an Hadamard space. A sequence $\left(f^n\right)$ of functions $f^n:\hs\to\exrls$ is said to $\Gamma$-converge to a~function $f:\hs\to\exrls$ if, for any $x\in \hs,$ we have
\begin{enumerate}
 \item[$(\Gamma1)$] $f(x)\leq \liminf_{n\to\infty} f^n(x_n),$ whenever $x_n\to x,$ and
 \item[$(\Gamma2)$]  there exists $(y_n)\subset \hs$ such that $y_n\to x,$ and $f^n(y_n)\to f(x).$
\end{enumerate}

Like in Hilbert spaces, $\Gamma$-convergence preserves convexity, and the limit function is always lsc~\cite{maso}. We will however use a stronger type of convergence, called the Mosco convergence.

The sequence $\left(f^n\right)$ converges to $f$ in the sense of Mosco if, for any $x\in \hs,$ we have
\begin{enumerate}
 \item[(M1)]  $f(x)\leq \liminf_{n\to\infty} f^n(x_n),$ whenever $x_n\wto x,$ and
 \item[(M2)] there exists $(y_n)\subset \hs$ such that $y_n\to x,$ and $f^n(y_n)\to f(x).$
\end{enumerate}

The advantage of the Mosco convergence is that it implies the convergence of the Moreau-Yosida envelopes and resolvents~\cite[Theorem~3.26]{attouch-b}. We now extend this result to Hadamard spaces.

We follow Attouch's original proofs of \cite[Theorem~3.26]{attouch-b} and \cite[Th\'eor\`eme~1.2]{attouch79} as much as possible, but at some places we have to use  different techniques to overcome nonlinearity of the space. In particular, we do not have tools like the inner product, or convex subdifferential in general Hadamard spaces. 

\begin{thm} \label{thm:mosco} Let $(\hs,d)$ be an Hadamard space, and $f^n:\hs\to\exrls$ be convex lsc functions, for $n\in\nat.$ If $f^n\to f$ in the sense of Mosco, as $n\to\infty,$ then
\begin{align}
 \lim_{n\to\infty} f_\lambda^n(x) &= f_\lambda(x),   \label{i:mosco:i} \\
\intertext{and,}
 \lim_{n\to\infty} J_\lambda^n(x) &= J_\lambda(x),  \label{i:mosco:ii}  
\end{align}
for any $\lambda>0,$ and $x\in \hs.$ Here $J_\lambda^n$ stands for the resolvent of $f^n$ and $J_\lambda$ for the resolvent of $f.$
\end{thm}
\begin{proof}
We will first show that the sequence $\left(J_\lambda^n x\right)_n$ is bounded. To do so, need the following claim; its linear version appeared in \cite[Lemme~1.5]{attouch79}.

\textbf{Claim:} Given $x_0\in \hs,$ there exist $\alpha,\beta>0$ such that
\begin{equation}
 f^n(x)\geq -\alpha d(y,x_0)-\beta,\quad\text{for any } y\in \hs,\,n\in\nat. \label{eq:support}
\end{equation}
Indeed, assume that this is not the case, that is, for any $k\in\nat,$ there exist $n_k\in\nat,$ and $x_k\in \hs$ such that
$$f^{n_k}(x_k)+k\left[d\left(x_k,x_0\right)+1\right]<0.$$
Without loss of generality we may assume $n_k\to\infty,$ as $k\to\infty.$ If $(x_k)$ were bounded, then there exist $\ol{x}\in \hs$ and a subsequence of $(x_k),$ still denoted $(x_k),$ such that $x_k\wto\ol{x}.$ By the Mosco convergence of $(f^n)$ we have 
$$f\left(\ol{x}\right)\leq\liminf_{k\to\infty} f^{n_k}(x_k)\leq-\limsup_{k\to\infty} k\left[d\left(x_k,x_0\right)+1\right]\leq-\infty,$$
which is impossible. Assume therefore $(x_k)$ is unbounded. Choose $y_0\in \hs$ and find $y_k\to y_0$ such that $f^{n_k}(y_k)\to f(y_0).$ Put
$$z_k=(1-t_k)y_k + t_k x_k,\quad\text{with } t_k=\frac1{\sqrt{k}d(x_k,y_k)}.$$
Then $z_k\to y_0.$ By convexity,
\begin{align*}
 f^{n_k}(z_k) & \leq (1-t_k)f^{n_k}(y_k)+ t_k f^{n_k}(x_k) \\
        & \leq (1-t_k)f^{n_k}(y_k) -t_k k\left[d(x_k,x_0)+1\right]\\
        & \leq (1-t_k)f^{n_k}(y_k)-\sqrt{k}\frac{d(x_k,x_0)+1}{d(x_k,y_k)}.
\end{align*}
Hence,
$$f(y_0)\leq\liminf_{k\to\infty} f^{n_k}(z_k)\leq -\infty,$$
which is not possible, too. This proves the \textbf{claim.} 

For any $n\in\nat$ we hence have
$$f^n\left(J_\lambda^n x\right)\geq -\alpha d(J_\lambda^n x,x_0)-\beta.$$
Choose a sequence $(u_n)\subset \hs$ such that $u_n\to x_0$ and $f^n(u_n)\to f(x_0).$ From the definition of $J_\lambda^n x,$ we have
$$f^n(u_n)+\frac1{2\lambda} d(x,u_n)^2\geq f^n\left(J_\lambda^n x\right)+\frac1{2\lambda}d\left(x,J_\lambda^n x\right)^2,$$
and furthermore,
$$f^n(u_n)+\alpha d(J_\lambda^n x,x_0)+\beta+\frac1{2\lambda} d(x,u_n)^2\geq\frac1{2\lambda}d\left(x,J_\lambda^n x\right)^2,$$
which implies that the sequence $\left(J_\lambda^n x\right)_n$ is bounded.

Let $c\in \hs$ be a weak cluster point of $\left(J_\lambda^n x\right)_n.$ Its existence is guaranteed by boundedness of the sequence. Since $f^n\to f$ in the sense of Mosco, there exists a sequence $(y_n)\subset \hs$ such that $y_n\to J_\lambda x,$ and $f^n(y_n)\to f\left(J_\lambda x\right).$ Then 
\begin{align}
 \limsup_{n\to\infty} f_\lambda^n(x) & \leq  \limsup_{n\to\infty} \left[ f^n(y_n)+\frac1{2\lambda}d(x,y_n)^2 \right] \nonumber \\ 
             & = f\left(J_\lambda x\right)+\frac1{2\lambda} d\left(x,J_\lambda x\right)^2  \nonumber \\ 
             & \leq f(c)+\frac1{2\lambda} d\left(x,c\right)^2 \nonumber \\ &\leq \liminf_{n\to\infty} \left[f^n\left(J_\lambda^n x\right)+\frac1{2\lambda} d\left(x,J_\lambda^n x\right)^2 \right], \label{eq:a}
\end{align}
which gives $J_\lambda x=c,$ by the uniqueness of $J_\lambda x.$ Hence, since $c$ was arbitrary, the whole sequence $\left(J_\lambda^n x\right)_n$ weakly converges to $J_\lambda x.$ Furthermore,
\begin{align}
\limsup_{n\to\infty}\frac1{2\lambda} d\left(x,J_\lambda^n x\right)^2 & \leq  \limsup_{n\to\infty}\left(-f^n\left(J_\lambda^nx\right)\right)+ \limsup_{n\to\infty}f^n\left( y_n\right) \nonumber \\ & \quad + \limsup_{n\to\infty}\frac1{2\lambda}d(x,y_n)^2 \nonumber \\
  & \leq - \liminf_{n\to\infty}f^n\left(J_\lambda^nx\right)+f\left(J_\lambda x\right)+ \frac1{2\lambda}d(x,J_\lambda x)^2\nonumber  \\ & \leq\frac1{2\lambda}d(x,J_\lambda x)^2 \nonumber\\ & \leq \liminf_{n\to\infty}\frac1{2\lambda} d\left(x,J_\lambda^n x\right)^2 . \label{eq:dd}
\end{align}
Lemma~\ref{lem:wtos} and~\eqref{eq:dd} give together the strong convergence
$$J_\lambda^n x \to J_\lambda x,\quad\text{as } n\to\infty,$$
which proves~\eqref{i:mosco:ii}. Finally, via~\eqref{eq:a} we get
$$\lim_{n\to\infty} f_\lambda^n(x)=\lim_{n\to\infty}\left[f^n\left(J_\lambda^nx\right)+\frac1{2\lambda}d\left(x,J_\lambda^nx\right)^2\right]=f\left(J_\lambda x\right)+\frac1{2\lambda}d\left(x,J_\lambda x\right)^2=f_\lambda(x).$$
This gives~\eqref{i:mosco:i} and the proof is complete.
\end{proof}

A sequence of closed convex sets $C_n\subset\hs$ is said to \emph{Mosco converge} to a closed convex set $C\subset\hs$ if the indicator functions $\iota_{C_n}$ converge in the sense of Mosco to the indicator function $\iota_C.$
\begin{cor}\label{cor:wijsman}
Let $(\hs,d)$ be an Hadamard space. Assume $C,C_n\subset \hs$ are convex closed sets for each $n\in\nat.$ If the sequence $\left(C_n\right)$ Mosco converges to~$C,$ then
\begin{enumerate}
 \item $d\left(x,C_n\right)\to d(x,C),$ \label{i:wij}
 \item $P_{C_n}(x)\to P_C(x),$
\end{enumerate}
for any $x\in \hs.$
\end{cor}
\begin{proof}
 Follows immediately from the fact that the Moreau-Yosida envelope (with $\lambda=\frac12$) of the indicator function is the distance function squared, and the resolvent is the nearest point mapping. 
\end{proof}
The convergence in~Corollary~\ref{cor:wijsman}\eqref{i:wij} is called Frol\'ik-Wijsman~\cite{frolik,wijsman}. Hence the Mosco convergence of convex closed sets implies the Frol\'ik-Wijsman convergence. See also~\cite{tsukada}.

\begin{rem}
Let $\left(C_n\right)$ be a sequence of convex closed subsets of an Hadamard space. If $\left(C_n\right)$ is decreasing, then it converges in the sense of Mosco to its intersection. Likewise, if $\left(C_n\right)$ is increasing, it converges in the sense of Mosco to the closure of its union. These facts are rather straightforward to prove, see~\cite[Lemma~1.2, Lemma~1.3]{mosco69} for the linear case.
\end{rem}

We next show that the convergence of the resolvents implies the convergence of the semigroups. In linear spaces, the proof uses many tools which are not available in our setting \cite[Thm.~3.16, Thm. 4.2]{brezis-b}. We rather use quite a direct approach based on the slope estimates from Section~\ref{subsec:prelim}, which was employed by Stojkovic in his proof of the Trotter-Kato formula~\cite[Theorem~3.12]{stoj}. See also~\cite{brepaz,crandall}.
\begin{thm}\label{thm:semigr}
Let $(\hs,d)$ be an Hadamard space. Assume $f:\hs\to\exrls$ and $f^n:\hs\to\exrls,$ for $n\in\nat,$ are lsc convex functions. Let $J_\lambda,$ and $J_\lambda^n$ be the corresponding resolvents, with $\lambda>0,$ and let $\left(S_t\right)_{t>0}$ and $\left(S_t^n\right)_{t>0}$ be the corresponding semigroups. Assume that for any $x\in \cldom f,$ and $\lambda>0$ we have 
\begin{equation} \label{eq:rescon}
\lim_{n\to\infty} J_\lambda^n x= J_\lambda x.
\end{equation}
Then
\begin{equation} \label{eq:semcon}
\lim_{n\to\infty} S_t^n x= S_t x,
\end{equation}
for any $x\in \cldom f,$ and $t>0.$
\end{thm}
\begin{proof}
Assume first $x\in\dom |\partial f|.$ We then have
\begin{align*}
 d\left(S_t^n x,S_tx\right) & \leq d\left(S_t^n x, S_t^n \left(J_\lambda^n x\right)\right)+d \left(S_t^n\left(J_\lambda^n x\right), S_tx\right) \\
                            & \leq d\left(x,J_\lambda^n x\right) +d \left(S_t^n\left(J_\lambda^n x\right), S_tx\right),
\end{align*}
for any $\lambda>0$ and $n\in\nat.$ The second term on the right hand side can be further estimated
\begin{align*}
 d \left(S_t^n\left(J_\lambda^n x\right), S_tx\right)  \leq & d \left(S_t^n\left(J_\lambda^n x\right), \left(J_{\frac{t}{k}}^n\right)^{(k)} J_\lambda^n x\right)+d\left(\left(J_{\frac{t}{k}}^n\right)^{(k)} J_\lambda^n x,\left(J_{\frac{t}{k}}^n\right)^{(k)} x\right) \\
 & + d\left(\left(J_{\frac{t}{k}}^n\right)^{(k)} x,\left(J_{\frac{t}{k}}\right)^{(k)} x\right) + d\left(\left(J_{\frac{t}{k}}\right)^{(k)} x, S_t x\right).
\end{align*}
By~\eqref{eq:error} and~\eqref{eq:slope},
$$d \left(S_t^n\left(J_\lambda^n x\right), \left(J_{\frac{t}{k}}^n\right)^{(k)} J_\lambda^n x\right)\leq \frac{t}{\sqrt2k}\left|\partial f^n\right|\left(J_\lambda^n x\right)\leq \frac{t}{\sqrt2k}\frac{d\left(J_\lambda^n x,x\right)}{\lambda}.$$
Furthermore, the inequality~\eqref{eq:error} also yields
$$ d\left( \left(J_{\frac{t}{k}}\right)^{(k)} x, S_t x\right)\leq \frac{t}{\sqrt2k}\left|\partial f\right|(x).$$
Altogether we obtain,
\begin{align*}
d\left(S_t^n x,S_tx\right)  \leq & 2d\left(x,J_\lambda^n x\right)+\frac{t}{\sqrt2k}\frac{d\left(J_\lambda^n x,x\right)}{\lambda}\\ & + d\left(\left(J_{\frac{t}{k}}^n\right)^{(k)} x,\left(J_{\frac{t}{k}}\right)^{(k)} x\right) + \frac{t}{\sqrt2k}\left|\partial f\right|(x).
\end{align*}
Now fix $\eps>0$ and choose $\lambda_0\in(0,1)$ so that $\sqrt{\lambda_0}\left|\partial f\right|(x)<\eps.$ By the assumption~\eqref{eq:rescon}, and~\eqref{eq:slope2}, we have
$$ \lim_{n\to\infty}\frac{d\left(x,J_\lambda^n x\right)}{\sqrt{\lambda}} =\frac{d\left(x,J_\lambda x\right)}{\sqrt{\lambda}}\leq |\partial f|(x),$$
for any $\lambda>0.$ There is therefore $n_0\in\nat$ such that for all $n>n_0$ we have
$$\frac{d\left(x,J_{\lambda_0}^n x\right)}{\sqrt{\lambda_0}}\leq \frac{d\left(x,J_{\lambda_0} x\right)}{\sqrt{\lambda_0}}+\eps,$$
and hence
$$d\left(x,J_{\lambda_0}^n x\right)\leq \sqrt{\lambda_0}\left|\partial f\right|(x)+\eps\sqrt{\lambda_0}<2\eps.$$
Next choose $k_0\in\nat$ such that
$$ \frac{t}{\sqrt{2}k_0}\left|\partial f\right|(x) <\eps,$$
and simultaneously,
$$ \frac{t}{\sqrt{2}k_0}\frac{d\left(x,J_{\lambda_0}^n x\right)}{\lambda_0}<\eps,$$
for any $n>n_0.$ Then we can find $n_1>n_0$ so that for any $n>n_1$ we have
$$ d\left(\left(J_{\frac{t}{k_0}}^n\right)^{(k_0)} x,\left(J_{\frac{t}{k_0}}\right)^{(k_0)} x\right)<\eps .$$
Altogether we have
$$d\left(S_t^n x,S_tx\right)<7\eps,$$
for all $n>n_1,$ and $t\in[0,T].$ 

Let finally $x\in\cldom f.$ Since $\dom |\partial f|$ is dense in $\cldom f$ by Proposition~\ref{prop:slope}, there exists, for any $\eps>0,$ a point $y\in\dom |\partial f|$ such that
$d(x,y)<\eps.$ Then
$$d\left(S_t^n x,S_tx\right)\leq d\left(S_t^n x,S_t^n y\right)+d\left(S_t^n y,S_t y\right)+d\left(S_t y,S_tx\right)<2\eps+d\left(S_t^n y,S_ty\right),$$
which finishes the proof.
\end{proof}

\subsection{Convergence results on asymptotic relations}

The results presented here are only variants of Theorems~\ref{thm:mosco} and~\ref{thm:semigr}. They however significantly improve~\cite[Proposition~5.12]{japan}, and answer~\cite[Problem~5.26]{japan}, respectively.
 
Let $(\hs,d)$ be an Hadamard space and $\left(\hs^n,d^n\right)_{n\in\nat}$ be a sequence of Hadamard spaces, and suppose an asymptotic relation is given.
A sequence $\left(f^n\right)$ of functions $f^n:\hs^n\to\exrls$ is said to $\Gamma$-converge to a function $f:\hs\to\exrls$ if, for any $x\in \hs,$ we have
\begin{enumerate}
 \item[$(\Gamma1)$] $f(x)\leq \liminf_{n\to\infty} f^n(x_n),$ whenever $x_n\in \hs^n,$ and $x_n\to x,$ and
 \item[$(\Gamma2)$] there exists $y_n\in \hs^n$ such that $y_n\to x,$ and $f^n(y_n)\to f(x).$
\end{enumerate}
The sequence $\left(f^n\right)$ converges to $f$ in the sense of Mosco if, for any $x\in \hs,$ we have
\begin{enumerate}
 \item[(M1)] $f(x)\leq \liminf_{n\to\infty} f^n(x_n),$ whenever $x_n\in \hs^n,$ and $x_n\wto x,$ and
 \item[(M2)] there exists $y_n\in \hs^n$ such that $y_n\to x,$ and $f^n(y_n)\to f(x).$
\end{enumerate}

We first improve~\cite[Proposition~5.12]{japan} by removing the assumption of nonnegativity of the functions. This assumption surprisingly allowed for a~much easier proof in~\cite[Proposition~5.12]{japan}.
\begin{thm} \label{thm:mosco-ar} 
Let $(\hs,d)$ be an Hadamard space and $\left(\hs^n,d^n\right)_{n\in\nat}$ be a sequence of Hadamard spaces, and suppose an asymptotic relation is given. Let $f:\hs\to\exrls,$ and $f^n:\hs^n\to\exrls$ be convex lsc functions, for $n\in\nat.$ Let $x\in \hs,$ and choose $x_n\in \hs^n$ with $x_n\to x.$ If $f^n\to f,$ in the sense of Mosco, as $n\to\infty,$ then
\begin{align*}
\lim_{n\to\infty} f_\lambda^n(x_n) &= f_\lambda(x), \\
\intertext{and,}
\lim_{n\to\infty} J_\lambda^n(x_n) &= J_\lambda(x),
\end{align*}
for any $\lambda>0.$
\end{thm}
\begin{proof}
Choose $\lambda>0.$ We will first show that the sequence $\left(J_\lambda^n x_n\right)_n$ is bounded. We shall again use a version of~\cite[Lemme~1.5]{attouch79}.

\textbf{Claim:} For any $w\in \hs,$ there exist $\alpha,\beta>0$ such that for any $w_n\in \hs^n$ with $w_n\to w,$ we have
\begin{equation}
 f^n(v_n)\geq -\alpha d(v_n,w_n)-\beta,\quad\text{for any } v_n\in \hs^n,\,n\in\nat. \label{eq:support-ar}
\end{equation}
Indeed, assume that this is not the case, that is, for any $k\in\nat,$ there exist $n_k\in\nat,$ and $v_{n_k}\in \hs^{n_k}$ such that
$$f^{n_k}(v_{n_k})+k\left[d\left(v_{n_k},w_{n_k}\right)+1\right]<0.$$
If $(v_k)$ were bounded, then there exist $\ol{v}\in \hs$ and a subsequence of $\left(v_{n_k}\right),$ still denoted $\left(v_{n_k}\right),$ such that $v_{n_k}\wto\ol{v}.$ By the Mosco convergence of $(f^n)$ we have 
$$f\left(\ol{v}\right)\leq\liminf_{k\to\infty} f^{n_k}\left(v_{n_k}\right)\leq-\limsup_{k\to\infty} k\left[d\left(v_{n_k},w_{n_k}\right)+1\right]\leq-\infty,$$
which is impossible. Assume therefore $\left(v_{n_k}\right)$ is unbounded. Choose $y\in \hs$ and find $y_{n_k}\in \hs^{n_k}$ such that $y_{n_k}\to y$ and $f^{n_k}(y_{n_k})\to f(y).$ Put
$$z_{n_k}=\left(1-t_{n_k}\right)y_{n_k}+ t_{n_k} v_{n_k},\quad\text{with } t_{n_k}=\frac1{\sqrt{k}d\left(v_{n_k},y_{n_k}\right)}.$$
Then $z_k\to y.$ By convexity,
\begin{align*}
 f^{n_k}(z_k) & \leq \left(1-t_{n_k}\right)f^{n_k}\left(y_{n_k}\right)+ t_{n_k} f^{n_k}\left(v_{n_k}\right) \\
        & \leq \left(1-t_{n_k}\right)f^{n_k}\left(y_{n_k}\right) -t_{n_k} k\left[d\left(v_{n_k},w_{n_k}\right)+1\right]\\
        & \leq \left(1-t_{n_k}\right)f^{n_k}\left(y_{n_k}\right)-\sqrt{k}\frac{d\left(v_{n_k},w_{n_k}\right)+1}{d(v_{n_k},y_{n_k})}.
\end{align*}
Hence,
$$f(y)\leq\liminf_{k\to\infty} f^{n_k}\left(z_{n_k}\right)\leq -\infty,$$
which is not possible, too. This proves the \textbf{claim.} 

For any $n\in\nat$ we hence have
$$f^n\left(J_\lambda^n x\right)\geq -\alpha d(J_\lambda^n x,w_n)-\beta.$$
Choose a sequence $u_n\in \hs^n$ such that $u_n\to w$ and $f^n(u_n)\to f(w).$ From the definition of $J_\lambda^n x_n,$ we have
$$f^n(u_n)+\frac1{2\lambda} d(x_n,u_n)^2\geq f^n\left(J_\lambda^n x_n\right)+\frac1{2\lambda}d\left(x_n,J_\lambda^n x_n\right)^2,$$
and furthermore,
$$f^n(u_n)+\alpha d(J_\lambda^n x_n,w_n)+\beta+\frac1{2\lambda} d(x_n,u_n)^2\geq\frac1{2\lambda}d\left(x_n,J_\lambda^n x_n\right)^2,$$
which implies that the sequence $\left(J_\lambda^n x_n\right)_n$ is bounded.

Let $c\in \hs$ be a weak cluster point of $\left(J_\lambda^n x_n\right)_n.$ Its existence is guaranteed by boundedness of the sequence. Since $f^n\to f$ in the sense of Mosco, there exists a sequence $y_n\in \hs$ such that $y_n\to J_\lambda x,$ and $f^n(y_n)\to f\left(J_\lambda x\right).$ Then 
\begin{align}
 \limsup_{n\to\infty} f_\lambda^n(x_n) & \leq \limsup_{n\to\infty} \left[f^n(y_n)+\frac1{2\lambda}d(x_n,y_n)^2 \right]\nonumber \\ 
            & = f\left(J_\lambda x\right)+\frac1{2\lambda}d\left(x,J_\lambda x\right)^2 \nonumber \\ & \leq f(c)+\frac1{2\lambda} d\left(x,c\right)^2\nonumber \\ 
            & \leq\liminf_{n\to\infty}\left[f^n\left(J_\lambda^n x_n\right)+\frac1{2\lambda} d\left(x_n,J_\lambda^n x_n\right)^2\right], \label{eq:b}
\end{align}
which gives $J_\lambda x=c,$ by uniqueness of $J_\lambda x.$ Hence, since $c$ was arbitrary, the whole sequence $\left(J_\lambda^n x_n\right)_n$ weakly converges to $J_\lambda x.$ Furthermore,
\begin{align}
 \limsup_{n\to\infty}\frac1{2\lambda} d\left(x_n,J_\lambda^n x_n\right)^2 & \leq  \limsup_{n\to\infty}\left(-f^n\left(J_\lambda^nx_n\right)\right)+ \limsup_{n\to\infty}f^n\left( y_n\right) \nonumber \\ & \quad + \limsup_{n\to\infty}\frac1{2\lambda}d(x_n,y_n)^2 \nonumber \\
  & \leq - \liminf_{n\to\infty}f^n\left(J_\lambda^nx_n\right)+f\left(J_\lambda x\right)+ \frac1{2\lambda}d(x,J_\lambda x)^2\nonumber \\ & \leq\frac1{2\lambda}d(x,J_\lambda x)^2 \nonumber \\ &\leq \liminf_{n\to\infty}\frac1{2\lambda} d\left(x_n,J_\lambda^n x_n\right)^2 . \label{eq:ff}
\end{align}
Lemma~\ref{lem:weakar} and~\eqref{eq:ff} give strong convergence
$$ J_\lambda^n x_n \to J_\lambda x,\quad\text{as } n\to\infty.$$
Finally, via~\eqref{eq:b}, we get
\begin{align*}
 \lim_{n\to\infty} f_\lambda^n(x_n) & =\lim_{n\to\infty} \left[f^n\left(J_\lambda^nx_n\right)+\frac1{2\lambda}d\left(x_n,J_\lambda^nx_n\right)^2\right] \\
& =f\left(J_\lambda x\right)+\frac1{2\lambda}d\left(x,J_\lambda x\right)^2 \\ &=f_\lambda(x).
\end{align*}
The proof is complete.
\end{proof}

We now show that the convergence of the resolvents implies the convergence of the semigroups, which in combination with Theorem~\ref{thm:mosco-ar} gives that Mosco convergence of a sequence of functions implies strong convergence of the corresponding semigroups. This answers~\cite[Problem~5.26]{japan}. Note that the hitherto best result in this direction applies only to Hilbert spaces and quadratic forms~\cite[Theorem~5.27]{japan}.
\begin{thm}\label{thm:semigr-ar}
Let $(\hs,d)$ be an Hadamard space and $\left(\hs^n,d^n\right)_{n\in\nat}$ be a sequence of Hadamard spaces, and suppose an asymptotic relation is given.
Assume $f:\hs\to\exrls$ and $f^n:\hs^n\to\exrls,$ for $n\in\nat,$ are lsc convex functions. Let $J_\lambda: \hs\to \hs,$ and $J_\lambda^n:\hs^n\to \hs^n$ be the corresponding resolvents, with $\lambda>0,$ and let $S_t:\cldom f\to\dom f$ and $S_t^n:\cldom f^n\to\dom f^n$ be the corresponding semigroups, with $t>0.$ Assume that for any $x\in \cldom f,$ any $x_n\in \hs^n,$ with $x_n\to x,$ and any $\lambda>0,$ we have 
\begin{align} \label{eq:rescon-ar}
\lim_{n\to\infty} J_\lambda^n x_n &= J_\lambda x. \\
\intertext{Then,}
\label{eq:semcon-ar}
\lim_{n\to\infty} S_t^n x_n &=S_t x,
\end{align}
for any $x\in \cldom f,$ any $x_n\in \hs^n,$ with $x_n\to x,$ and any $t>0.$
\end{thm}
\begin{proof}
Assume first $x\in\dom |\partial f|$ and choose a sequence $z_n^t\in \hs^n$ such that $z_n^t\to S_tx.$ We then have
\begin{align*}
 d\left(S_t^n x_n,z_n^t\right) & \leq d\left(S_t^n x_n, S_t^n \left(J_\lambda^n x_n\right)\right)+d \left(S_t^n\left(J_\lambda^n x_n\right), z_n^t\right) \\
                            & \leq d\left(x_n,J_\lambda^n x_n\right) +d \left(S_t^n\left(J_\lambda^n x_n\right), z_n^t\right),
\end{align*}
for any $\lambda>0$ and $n\in\nat.$ For any $t\in[0,T]$ and $k\in\nat,$ find a sequence $y_n^{t,k}\in \hs^n$ with
$$y_n^{t,k}\to \left(J_{\frac{t}{k}}\right)^{(k)} x,\quad \text{as } n\to\infty.$$
The last term on the right hand side of the above inequality can be further estimated
\begin{align*}
 d \left(S_t^n\left(J_\lambda^n x_n\right), z_n^t\right) & \leq  d \left(S_t^n\left(J_\lambda^n x_n\right), \left(J_{\frac{t}{k}}^n\right)^{(k)} J_\lambda^n x_n\right) \\ & \quad +d\left(\left(J_{\frac{t}{k}}^n\right)^{(k)} J_\lambda^n x_n,\left(J_{\frac{t}{k}}^n\right)^{(k)} x_n\right) \\
 & \quad + d\left(\left(J_{\frac{t}{k}}^n\right)^{(k)} x_n,y_n^{t,k}\right) + d\left(y_n^{t,k}, z_n^t\right).
\end{align*}
By~\eqref{eq:error} and~\eqref{eq:slope},
$$d \left(S_t^n\left(J_\lambda^n x_n\right), \left(J_{\frac{t}{k}}^n\right)^{(k)} J_\lambda^n x_n\right)\leq \frac{t}{\sqrt2k}\left|\partial f^n\right|\left(J_\lambda^nx_n\right)\leq \frac{t}{\sqrt2k}\frac{d\left(J_\lambda^n x_n,x_n\right)}{\lambda}.$$
Furthermore, the inequality~\eqref{eq:error} also yields
$$ d\left( \left(J_{\frac{t}{k}}\right)^{(k)} x, S_t x\right)\leq \frac{t}{\sqrt2k}\left|\partial f\right|(x).$$
Then
\begin{align*}
d\left(S_t^n x_n,z_n^t\right)  \leq & 2d\left(x_n,J_\lambda^n x_n\right)+\frac{t}{\sqrt2k}\frac{d\left(J_\lambda^n x_n,x_n\right)}{\lambda}\\ & + d\left(\left(J_{\frac{t}{k}}^n\right)^{(k)} x_n,y_n^{t,k}\right) +d\left(y_n^{t,k}, z_n^t\right).
\end{align*}
Now fix $\eps>0$ and choose $\lambda_0\in(0,1)$ so that $\sqrt{\lambda_0}\left|\partial f\right|(x)<\eps.$ By the assumption~\eqref{eq:rescon-ar}, and~\eqref{eq:slope2}, we have
$$ \lim_{n\to\infty}\frac{d\left(x_n,J_\lambda^n x_n\right)}{\sqrt{\lambda}} =\frac{d\left(x,J_\lambda x\right)}{\sqrt{\lambda}}\leq |\partial f|(x),$$
for any $\lambda>0.$ There is therefore $n_0\in\nat$ such that for all $n>n_0$ we have
$$\frac{d\left(x_n,J_{\lambda_0}^n x_n\right)}{\sqrt{\lambda_0}}\leq \frac{d\left(x,J_{\lambda_0} x\right)}{\sqrt{\lambda_0}}+\eps,$$
and hence
$$d\left(x_n,J_{\lambda_0}^n x_n\right)\leq \sqrt{\lambda_0}\left|\partial f\right|(x)+\eps\sqrt{\lambda_0}<2\eps.$$
Now choose $k_0\in\nat$ such that
$$ \frac{t}{\sqrt{2}k_0}\left|\partial f\right|(x) <\eps,$$
and simultaneously,
$$ \frac{t}{\sqrt{2}k_0}\frac{d\left(x_n,J_{\lambda_0}^n x_n\right)}{\lambda_0}<\eps,$$
for any $n>n_0.$ Then we can find $n_1>n_0$ so that for any $n>n_1$ we have
$$d\left(y_n^{t,k_0}, z_n^t\right)<2\eps,$$
and simultaneously
$$ d\left(\left(J_{\frac{t}{k_0}}^n\right)^{(k_0)} x_n,y_n^{t,k_0}\right)<\eps .$$
Altogether we obtain
$$d\left(S_t^n x_n,z_n^t\right)<8\eps,$$
for all $n>n_1,$ and $t\in[0,T].$

Let finally $x\in\cldom f.$ Since $\dom |\partial f|$ is dense in $\cldom f$ by Proposition~\ref{prop:slope}, there exists, for any $\eps>0,$ a point $y\in\dom |\partial f|$ such that
$d(x,y)<\eps.$ Find $x_n,y_n\in \hs^n,$ such that $x_n\to x,$ and $y_n\to y.$ Then, for large $n,$ we get
$$d\left(S_t^n x_n,u_n\right)\leq d\left(S_t^n x_n,S_t^n y_n\right)+d\left(S_t^n y_n,v_n\right)+d\left(v_n,u_n\right)<4\eps+d\left(S_t^n y_n,v_n\right),$$
where $u_n,v_n\in \hs^n$ are sequences which converges to $S_tx$ and $S_ty,$ respectively. The proof is now complete.
\end{proof}


\bibliographystyle{siam}
\bibliography{semigroups}

\end{document}